\documentclass[11pt]{article}
\usepackage{amsthm,amstext}
\usepackage{amsmath}
\usepackage{graphicx}
\usepackage{amsfonts}
\usepackage{amssymb}
\theoremstyle{plain}
\newtheorem{theorem}{Theorem}[section]
\newtheorem{lemma}[theorem]{Lemma}
\newtheorem{proposition}[theorem]{Proposition}

\theoremstyle{definition}
\newtheorem{definition}[theorem]{Definition}
\newtheorem{corollary}[theorem]{Corollary}
\newtheorem{example}{\sc Example}
\theoremstyle{remark}
\newtheorem{remark}{\sc Remark}

\date{}
\title{\bf Atanassov's Intuitionistic Fuzzy Ideals of $\Gamma$-Semigroups}\vspace{.25 in}
\author{ {\bf Sujit Kumar Sardar$^1$,} {\bf Samit Kumar Majumder$^2$}\\
and\\
{\bf Manasi Mandal$^3$}\\
Department of Mathematics, Jadavpur\\
University, Kolkata-700032, INDIA\\
{\tt $^1$sksardarjumath@gmail.com}\\
{\tt $^2$samitfuzzy@gmail.com}\\
{\tt $^3$manasi$_{-}$ju@yahoo.in}
 }

\begin{document}
\maketitle

\begin{abstract}

The notion of intuitionistic fuzzy set was introduced by Atanassov as a generalization of the notion of fuzzy set. In this paper we apply this concept of Atanassov to ideals, prime ideals and semiprime ideals of $\Gamma$-semigroups in order to obtain some characterization theorems. We also introduce the notion of  Atanassov's intuitionistic fuzzy ideal extension in a $\Gamma$-semigroup and investigate some of their important properties. A regular $\Gamma$-semigroup has been characterized in terms of Atanasov's intutionistic fuzzy ideal. Characterization of prime ideal of a $\Gamma$-semigroup has also been obtained in terms of Atanassov's intutionistic fuzzy ideal extension. \\


\end{abstract}

\section{Introduction}
A semigroup is an algebraic structure consisting of a non-empty set $S$ together with an associative binary operation\cite{H}. The formal study of semigroups began in the early 20th century. Semigroups are important in many areas of mathematics, for example, coding and language theory, automata theory, combinatorics and mathematical analysis. In 1981 M.K. Sen\cite{Sen} introduced the notion of $\Gamma$-semigroup as a generalization of semigroup and ternary semigroup. We call this $\Gamma$-semigroup a {\it both sided $\Gamma$-semigroup.} In 1986 M.K. Sen and N.K. Saha\cite{SS}
modified the definition of Sen's $\Gamma$-semigroup. This newly defined $\Gamma$-semigroup is known as {\it one sided $\Gamma$-semigroup.} $\Gamma$-semigroups have been analyzed by a lot of mathematicians, for instance by Chattopadhyay\cite{Ch2}, Dutta and  Adhikari\cite{A,D1,D2,D3}, Hila\cite {H1,H2}, Chinram\cite{Chin}, Sen et al.\cite{SC,SSe}. T.K. Dutta and N.C. Adhikari\cite{A,D1} mostly worked on {\it both sided $\Gamma$-semigroups.} They defined operator semigroups of such type of $\Gamma$-semigroups and established many results and obtained many correspondences between a $\Gamma$-semigroup and its operator semigroups. In this paper we have considered {\it both sided $\Gamma$-semigroups.}\\
\indent After the introduction of fuzzy set by Zadeh\cite{Z}, reconsideration of the concept of classical mathematics began. On the other hand, because of the importance of group theory in mathematics, as well as its many areas of application, the notion of fuzzy subgroups was defined by Rosenfeld\cite{R} and its structure was investigated. Das characterized fuzzy subgroups by their level subgroups in \cite{Das}. Nobuaki Kuroki\cite{K1,K2,K3} is the pioneer of fuzzy ideal theory of semigroups. The idea of fuzzy subsemigroup was also introduced by Kuroki$\cite{K1,K3}$. In \cite{K2}, Kuroki characterized several classes of semigroups in terms of fuzzy left, fuzzy right and fuzzy bi-ideals. Others who worked on fuzzy semigroup theory, such as X.Y. Xie\cite{X2}, Y.B. Jun\cite{J1}, are mentioned in the bibliography.\\
\indent In $2007,$ Uckun Mustafa, Ali Mehmet and Jun Young Bae\cite{U} introduced the notion of intuitionistic fuzzy ideals in $\Gamma$-semigroups. Motivated by Kuroki\cite{K1,K2,K3}, Mustafa et al.\cite{U}, S.K. Sardar et al.\cite{S1,S2,S3} have initiated the study of $\Gamma$-semigroups in terms of fuzzy sets. The purpose of this paper is as mentioned in the abstract.

\section{Preliminaries}

After the introduction of fuzzy sets by Zadeh\cite{Z}, several researches were conducted on the generalization of fuzzy sets. As an important generalization of the notion of fuzzy sets on a non-empty set $X,$ Atanassov introduced in \cite{A1,A2} the concept of $IFS(X)$\footnote{$IFS(X)$ denote the intuitionistic fuzzy sets defined on a non-empty set $X.$} as objects having the form\\
$$A=\{<x,\mu_{A}(x),\nu_{A}(x)>:x\in X\},$$
where the functions $\mu_{A}: X\rightarrow [0,1]$ and $\nu_{A}: X\rightarrow [0,1]$ denote the degree of membership and the degree of non-membership of each element $x\in X$ to the set $A$ respectively, and $0\leq \mu_{A}(x)+\nu_{A}(x)\leq 1$ for all $x\in X.$\\

\indent Let $A$ and $B$ be two $IFS(X).$ Then the following expressions are defined in \cite{A1,A2}.\\

\indent $(1)$ $A\subseteq B$ if and only if $\mu_{A}(x)\leq\mu_{B}(x)$ and $\nu_{A}(x)\geq\nu_{B}(x),$\\

\indent $(2)$ $A=B$ if and only if $A\subseteq B$ and $B\subseteq A,$\\

\indent $(3)$ $A^{C}=\{<x,\nu_{A}(x),\mu_{A}(x)>:x\in X\},$\\

\indent $(4)$ $A\cap B=\{<x,\min\{\mu_{A}(x),\mu_{B}(x)\},\max\{\nu_{A}(x),\nu_{B}(x)\}>:x\in X\},$\\

\indent $(5)$ $A\cup B=\{<x,\max\{\mu_{A}(x),\mu_{B}(x)\},\min\{\nu_{A}(x),\nu_{B}(x)\}>:x\in X\},$\\

\indent $(6)$ $\square A=\{<x,\mu_{A}(x),1-\mu_{A}(x)>:x\in X\},$\\

\indent $(7)$ $\lozenge A=\{<x,1-\nu_{A}(x),\nu_{A}(x)>:x\in X\}.$\\

For a non-empty family of $IFS(X),$ $A_{i}=(\mu_{A_{i}},\nu_{A_{i}})_{i\in I}$ we define $\underset{i\in I}{\inf}$ $A_{i}=(\underset{i\in I}{\inf}$ $\mu_{A_{i}},\underset{i\in I}{\sup}$ $\nu_{A_{i}})$ and $\underset{i\in I}{\sup}$ $A_{i}=(\underset{i\in I }{\sup}$ $\mu_{A_{i}},\underset{i\in I}{\inf}$ $\nu_{A_{i}})$ are also the $IFS(X),$ given as $:$ $\underset{i\in I}{\inf}$ $A_{i}:X\rightarrow [0,1],x\rightarrow (\underset{i\in I}{\inf}$ $\mu_{A_{i}}(x),\underset{i\in I}{\sup}$ $\nu_{A_{i}}(x))$ and $\underset{i\in I}{\sup}$ $A_{i}:X\rightarrow [0,1],x\rightarrow (\underset{i\in I}{\sup}$ $\mu_{A_{i}}(x),\underset{i\in I}{\inf}$ $\nu_{A_{i}}(x)).$\\

For the sake of simplicity, we shall use the symbol $A=(\mu_{A},\nu_{A})$ for the $IFS(X),$ $A=\{<x,\mu_{A}(x),\nu_{A}(x)>:x\in X\}.$\\

Now we shall discuss some elementary concepts of $\Gamma$-semigroup theory which will be required in the sequel.\\

Let $S$ and $\Gamma$ be two non-empty sets. $S$ is called a $\Gamma$-semigroup\cite{D1} if there exist mappings from $S\times\Gamma\times S$ to $S,$ written as $(a,\alpha,b)\longrightarrow a\alpha b,$ and from $\Gamma\times S\times\Gamma$ to $\Gamma,$ written as $(\alpha,a,\beta)\longrightarrow\alpha a\beta$ satisfying the following associative laws $(a\alpha b)\beta c=a(\alpha b\beta)c=a\alpha(b\beta c)$ and $\alpha(a\beta b)\gamma=(\alpha a\beta)b\gamma=\alpha a(\beta b\gamma)$ for all $a,b,c\in S$ and for all $\alpha,\beta,\gamma\in\Gamma.$\\
\indent Let $S$ be the set of all integers of the form $4n+1$ and $\Gamma$ be the set of all integers of the form $4n+3$ where $n$ is an integer. If $a\alpha b$ is $a+\alpha+b$ and $\alpha a\beta$ is $\alpha
+a+\beta($usual sum of integers$)$ for all $a,b\in S$ and for all $\alpha,\beta\in\Gamma,$ then $S$ is a $\Gamma$-semigroup\cite{D1}.\\
\indent A $LI(S)(RI(S))$\cite{D1}\footnote{$LI(S)$, $RI(S)$, $I(S),$ $PI(S),$ $SPI(S)$ respectively denote the left ideal(s), right ideal(s), ideal(s), prime ideal(s), semiprime ideal(s) of a $\Gamma$-semigroup $S.$} is a non-empty subset $I$ of $S$ such that $S\Gamma I \subseteq I$ ($I \Gamma S \subseteq I$).
If $I$ is both a $LI(S)$ and a $RI(S),$ then we say that $I$ is an $I(S)$\cite{D1}.\\
Let $S$ be a $\Gamma$-semigroup. An $I(S),$ $P$ is said to be $PI(S)$\cite{D2} if, for any two $I(S),$ $A$ and $B,$ $A\Gamma B\subseteq P$ implies that $A\subseteq P$ or $B\subseteq P.$\\
Let $S$ be a $\Gamma$-semigroup. An $I(S),$ $Q$ is said to be $SPI(S)$\cite{D2} if, for any $I(S),$ $A,$ $A\Gamma A\subseteq Q$ implies that $A\subseteq Q.$\\
Now we recall the following theorem.
\begin{theorem}
$\cite{S2,S3}$ Let $I$ be an $I(S).$ Then the following are equivalent:\\

$(1)$ $I$ is $PI(S)(SPI(S)),$\\

$(2)$ $x\Gamma S\Gamma y\subseteq I\Rightarrow x\in I$ or $y\in I($resp. $x\Gamma S\Gamma x\subseteq I\Rightarrow x\in I),$\\

$(3)$ $x\Gamma y\subseteq I\Rightarrow x\in I$ or $y\in I($resp. $x\Gamma x\subseteq I\Rightarrow x\in I).$
\end{theorem}

\section{Intuitionstic Fuzzy Ideals}

Unless or otherwise stated throughout this paper $S$ stands for a both sided $\Gamma$-semigroup.\\

\begin{definition}
\cite{U} A non-empty $IFS,$ $A=(\mu_{A},\nu_{A})$\ of $S$\ is called an $IFLI(S)$\footnote{$IFS(S)$, $IFLI(S)$, $IFRI(S)$, $IFI(S)$ denote respectively the intuitionistic fuzzy subset(s), intuitionistic fuzzy left ideal(s), intuitionistic fuzzy right ideal(s), intuitionistic fuzzy ideal(s) of a $\Gamma$-semigroup $S.$} if it satisfies:\\

\indent $(1)$ $\mu_{A} (x\gamma y)\geq \mu_{A} (y)\ {\rm for \ all} \  x,y\in S \ {\rm and} \ {\rm for \ all}\ \gamma \in \Gamma,$\\

\indent $(2)$ $\nu_{A} (x\gamma y)\leq \nu_{A} (y) \ {\rm for \ all} \  x,y\in S \ {\rm and} \ {\rm for \ all}\ \gamma \in \Gamma.$
\end{definition}

\begin{definition}
\cite{U} A non-empty $IFS,$ $A=(\mu_{A},\nu_{A})$\ of $S$\ is called an $IFRI(S)$ if it satisfies:\\

\indent $(1)$ $\mu_{A} (x\gamma y)\geq \mu_{A} (x)\ {\rm for \ all} \  x,y\in S \ {\rm and} \ {\rm for \ all}\ \gamma \in \Gamma,$\\

\indent $(2)$ $\nu_{A} (x\gamma y)\leq \nu_{A} (x) \ {\rm for \ all} \  x,y\in S \ {\rm and} \ {\rm for \ all}\ \gamma \in \Gamma.$
\end{definition}

\begin{definition}
\cite{U} A non-empty $IFS,$ $A=(\mu_{A},\nu_{A})$\ of $S$\ is called an $IFI(S)$ if it is an $IFLI(S)$ and $IFRI(S).$
\end{definition}

\begin{example} Let $S$ be the set of all non-positive integers and $\Gamma $ be the set of all non-positive even integers. Then $S$ is a $\Gamma $-semigroup where $a\gamma b$ denote the usual multiplication of integers $a,\gamma ,b$  with $a,b\in S$ and $\gamma \in \Gamma .$ Let $A=(\mu_{A},\nu_{A}) $ be an $IFS$ of $S,$ defined
as follows
$$
\mu_{A} (x)=\left\{
\begin{array}{ll}
1 & \text{if} \ x=0 \\
0.1 & \text{if} \ x=-1,-2 \\
0.2& \text{if} \ x<-2
\end{array}
\right.
$$
$$
\nu_{A} (x)=\left\{
\begin{array}{ll}
0 & \text{if} \ x=0 \\
0.6 & \text{if} \ x=-1,-2 \\
0.7 & \text{if} \ x<-2
\end{array}
\right. .
$$
Then the $IFS,$ $A=(\mu_{A},\nu_{A}) $ of $S$ is an $IFI(S).$
\end{example}

For all the results formulated in this paper, we only describe proof for the $IFLI.$ For $IFRI$ similar results hold as well.

By routine verification we obtain the following proposition and subsequent lemmas.

\begin{proposition}
If $\{A_{i}\}_{i\in \Lambda}$ is a family of $IFLI(S)(IFRI(S),IFI(S)),$ then $\underset{i\in\Lambda}\bigcap A_{i}$ is an $IFLI(S)(IFRI(S),IFI(S)).$
\end{proposition}

\begin{lemma}
If $A=(\mu_{A},\nu_{A})$ is an $IFLI(S)(IFRI(S),IFI(S)),$ then so is $\square A=(\mu_{A},\mu^{c}_{A}).$
\end{lemma}

\begin{lemma}
If $A=(\mu_{A},\nu_{A})$ is an $IFLI(S)(IFRI(S),IFI(S)),$ then so is $\lozenge A=
(\nu^{c}_{A},\nu_{A}).$
\end{lemma}

Combining Lemmas $3.5$ and $3.6$ we obtain the following theorem.

\begin{theorem}
$A=(\mu_{A},\nu_{A})$ is an $IFLI(S)(IFRI(S),IFI(S)),$ if and only if $\square A$ and $\lozenge A$ are $IFLI(S)(IFRI(S),IFI(S)).$
\end{theorem}

\begin{definition}
For any $t\in [0,1]$ and a fuzzy subset $\mu$ of $S,$ the set
$$
\begin{array}{ll}
U(\mu; t)=\{x\in S:\mu(x)\geq t\}(\text{resp. } L(\mu; t)=\{x\in S:\mu(x)\leq t\})
\end{array}
$$
is called an \it{upper$($resp. lower$)$ $t$-level cut} of $\mu.$
\end{definition}

\begin{theorem}
If $A=(\mu_{A},\nu_{A})$ is an $IFLI(S)(IFRI(S),IFI(S)),$ then the upper and lower level cuts $U(\mu_{A};t)$ and $L(\nu_{A};t)$ are $LI(S)(RI(S),I(S)),$ for every $t\in Im(\mu_{A})\cap Im(\nu_{A}).$
\end{theorem}

\begin{proof}
Let $t\in Im(\mu_{A})\cap Im(\nu_{A}).$ Let $x\in S,\gamma\in\Gamma$ and $y\in U(\mu_{A};t).$ Then $\mu_{A}(y)\geq t.$ Since $A=(\mu_{A},\nu_{A})$ is an $IFLI(S),$ hence $\mu_{A}(x\gamma y)\geq \mu_{A}(y)\geq t.$ Consequently, $x\gamma y\in U(\mu_{A};t).$

Again, let $x\in S,\gamma\in\Gamma$ and $y\in L(\nu_{A};t).$ Then $\nu_{A}(y)\leq t.$ Since $A=(\mu_{A},\nu_{A})$ is an $IFLI(S),$ hence $\nu_{A}(x\gamma y)\leq \nu_{A}(y)\leq t.$ Consequently, $x\gamma y\in L(\nu_{A};t).$ Hence $U(\mu_{A};t)$ and $L(\nu_{A};t)$ are $LI(S).$ Similarly we can prove the other cases also.
\end{proof}

\begin{theorem}
If $A=(\mu_{A},\nu_{A})$ is an $IFS$ of $S$ such that the non-empty sets $U(\mu_{A};t)$ and $L(\nu_{A};t)$ are $LI(S)(RI(S),I(S)),$ for $t\in [0,1].$ Then the $IFS,$ $A=(\mu_{A},\nu_{A})$ is an $IFLI(S)(IFRI(S),IFI(S)).$
\end{theorem}

\begin{proof}
For $t\in [0,1],$ let us assume that the non-empty sets $U(\mu_{A};t)$ and $L(\nu_{A};t)$ are $LI(S).$ Now we shall show that $A=(\mu_{A},\nu_{A})$ satisfies the conditions of Definition $3.1.$ Let $y\in S$ such that $\mu_{A}(y)=t_{0}=\nu_{A}(y).$ Then $y\in U(\mu_{A};t_{0})$ and $y\in L(\nu_{A};t_{0}).$ Let $x\in S$ and $\gamma\in\Gamma.$ Since $U(\mu_{A};t)$ and $L(\nu_{A};t)$ are $LI(S),$ so $x\gamma y\in U(\mu_{A};t_{0})$ and $x\gamma y\in L(\nu_{A};t_{0})$ which implies that $\mu_{A}(x\gamma y)\geq t_{0}=\mu_{A}(y)$ and $\nu_{A}(x\gamma y)\leq t_{0}=\nu_{A}(y).$ Hence $A=(\mu_{A},\nu_{A})$ is an $IFLI(S).$ Similarly we can prove the other cases also.
\end{proof}

\begin{proposition}
Let $S$ be a $\Gamma$-semigroup and $A=(\mu_{A},\nu_{A})$ be an $IFI(S).$ 

$(1)$ If $\omega$ be a fixed element of $S,$ then the set $A^{\omega}=(\mu^{\omega}_{A},\nu^{\omega}_{A})=\{x\in S: \mu_{A}(x)\geq\mu_{A}(\omega),\nu_{A}(x)\leq\nu_{A}(\omega)\}$ is an $I(S).$ 

$(2)$ The set $U=(\mu_{U},\nu_{U})=\{x\in S: \mu_{U}(x)=\mu_{U}(0),\nu_{U}(x)=\nu_{U}(0)\}$ is an $I(S).$
\end{proposition}

\begin{proof}
The proof is straightforward and so we omit it.
\end{proof}

\begin{theorem}
Let $I$ be a non-empty subset of $S.$ If two fuzzy
subsets $\mu$ and $\nu$ are defined on $S$ by $$\mu (x):=\left\{
\begin{array}{l}
\alpha _{0}\text{ \ if }x\in I \\
\alpha _{1}\text{ \ if }x\in S-I%
\end{array}%
\right. $$ and $$\nu (x):=\left\{
\begin{array}{l}
\beta _{0}\text{ \ if }x\in I \\
\beta _{1}\text{ \ if }x\in S-I%
\end{array}%
\right. $$
where $0\leq\alpha _{1}<\alpha _{0}, 0\leq\beta _{0}<\beta _{1}$ and $\alpha_{i}+\beta_{i}\leq 1$ for $i=0,1.$  Then $A=(\mu,\nu)$ is an $IFLI(S)(IFRI(S),IFI(S))$ and $U(\mu;\alpha_{0})=I=L(\nu;\beta_{0}).$
\end{theorem}

\begin{proof}
Let $I$ be a $LI(S)$ and $x,y\in S,\gamma\in\Gamma.$ If $y\notin I,$ then $\mu(y)=\alpha_{1}$ and $\nu(y)=\beta_{1}$ and $\mu(x\gamma y)=\alpha_{0}$ or $\alpha_{1}$ and $\nu(x\gamma y)=\beta_{0}$ or $\beta_{1}$ according as $x\gamma y\in I$ or $x\gamma y\notin I.$ Again if $y\in I,$ then $x\gamma y\in I$ and so $\mu(x\gamma y)=\alpha_{0}$ and $\nu(x\gamma y)=\beta_{0}.$ Thus, we see that $\mu(x\gamma y)\geq\mu(y)$ and $\nu(x\gamma y)\leq\nu(y).$ Consequently, $A=(\mu,\nu)$ is an $IFLI(S).$ Similar is the proof for other cases.

\indent In order to prove the converse, we first observe that by definition of $\mu$ and $\nu,$ $U(\mu;\alpha_{0})=I=L(\nu;\beta_{0}).$ Then the proof follows from Theorem $3.9.$
\end{proof}

Following result is the characteristic function criterion of an $IFLI(S)(IFRI(S),IFI\\(S))$ which follows as an easy consequence of the above result.

\begin{corollary}
Let $\chi_{P}$ be the characteristic function of a $LI(S)(RI(S),I(S)),$ $P.$ Then $I=(\chi_{P},\chi_{P}^{c})$ is an $IFLI(S)(IFRI(S),IFI(S)).$
\end{corollary}

Now we define composition of $IFI(S)$ in order to characterize regular $\Gamma$-semigroups in terms of $IFI(S).$

\begin{definition}
Let $S$ be a $\Gamma$-semigroup. Let $A=(\mu_{A},\nu_{A})$ and $B=(\mu_{B},\nu_{B})\in IFLI(S)[IFRI(S),IFI(S)].$ Then the product $A\circ B$ of $A$ and $B$ is defined as
\end{definition}

$(\mu_{A\circ B})(x)=\left\{
\begin{array}
[l]{l}%
\underset{x=u\gamma v}{\sup}[\min\{\mu_{A}(u),\mu_{B}(v)\}:u,v\in S;\gamma
\in\Gamma]\\
\ \ \ 0,\text{ if for any }u,v\in S\text{ and for any }\gamma\in\Gamma,x\neq u\gamma
v
\end{array}
\right.  $

and

$(\nu_{A\circ B})(x)=\left\{
\begin{array}
[l]{l}%
\underset{x=u\gamma v}{\inf}[\max\{\nu_{A}(u),\nu_{B}(v)\}:u,v\in S;\gamma
\in\Gamma]\\
\ \ \ 1,\text{ if for any }u,v\in S\text{ and for any }\gamma\in\Gamma,x\neq u\gamma
v
\end{array}
\right.  $

\begin{theorem}
In a $\Gamma$-semigroup $S$ the following are equivalent: $(1)$ $A=(\mu_{A},\nu_{A})$ is an $IFLI(S)(IFRI(S),IFI(S)),$ \ \ $(2)$ $S\circ A\subseteq A(A\circ S\subseteq A),$ where $S=(\chi_{S},\chi^{c}_{S})$ and $\chi_{S}$ is the characteristic function of $S.$
\end{theorem}

\begin{proof}
Let $A=(\mu_{A},\nu_{A})$ be an $IFLI(S)$. Let $a\in S.$ Suppose there exist
$u,v\in S$ and $\delta\in\Gamma$ such that $a=u\delta v.$ Then, since $A=(\mu_{A},\nu_{A})$ is
an $IFLI(S),$ we have%
\begin{align*}
(\chi_{S}\circ\mu_{A})(a)  &  =\underset{a=x\gamma y}{\sup}[\min\{\chi_{S}(x),\mu_{A}(y)\}]\\
&  =\underset{a=x\gamma y}{\sup}[\min\{1,\mu_{A}(y)\}]=\underset{a=x\gamma y}%
{\sup}{\mu_{A}(y)}
\end{align*}
and%
\begin{align*}
(\chi^{c}_{S}\circ\nu_{A})(a)  &  =\underset{a=x\gamma y}{\inf}[\max\{\chi^{c}_{S}(x),\nu_{A}(y)\}]\\
&  =\underset{a=x\gamma y}{\inf}[\max\{0,\nu_{A}(y)\}]=\underset{a=x\gamma y}%
{\inf}{\nu_{A}(y)}.
\end{align*}
Now, since $A=(\mu_{A},\nu_{A})$ is an $IFLI(S),$ $\mu_{A}(x\gamma y)\geq\mu_{A}(y)$ and $\nu_{A}(x\gamma y)\leq\nu_{A}(y)$ for all
$x,y\in S$ and for all $\gamma\in\Gamma.$ So in particular, $\mu_{A}(y)\leq\mu_{A}(a)$ and $\nu_{A}(y)\geq\nu_{A}(a)$ for all $a=x\gamma y$. Hence $\underset{a=x\gamma y}{\sup}{\mu_{A}(y)}\leq\mu_{A}(a)$ and $\underset{a=x\gamma y}{\inf}{\nu_{A}(y)}\geq\nu_{A}(a).$
Thus $\mu_{A}(a)\geq(\chi_{S}\circ\mu_{A})(a)$ and $\nu_{A}(a)\leq(\chi^{c}_{S}\circ\nu_{A})(a).$ If there do not exist $x,y\in S,\gamma
\in\Gamma$ such that $a=x\gamma y$ then $(\chi_{S}\circ\mu_{A})(a)=0\leq\mu_{A}(a)$ and $(\chi^{c}_{S}\circ\nu_{A})(a)=1\geq\nu_{A}(a).$ By a similar argument we can prove the other case also.

Conversely, Let $x,y\in S$, $\gamma\in\Gamma$
and $a:=x\gamma y.$ Then $\mu_{A}(x\gamma y)=\mu_{A}(a)\geq(\chi_{S}\circ\mu_{A})(a)$ and $\nu_{A}(x\gamma y)=\nu_{A}(a)\leq(\chi^{c}_{S}\circ\nu_{A})(a).$ Now%
\begin{align*}
(\chi_{S}\circ\mu_{A})(a)  &  =\underset{a=u\alpha v}{\sup}[\min\{\chi_{S}(u),\mu_{A}
(v)\}]\geq\min\{\chi_{S}(x),\mu_{A}(y)\}\\
&  =\min\{1,\mu_{A}(y)\}=\mu_{A}(y)
\end{align*}
and%
\begin{align*}
(\chi^{c}_{S}\circ\nu_{A})(a)  &  =\underset{a=u\alpha v}{\inf}[\max\{\chi^{c}_{S}(u),\nu_{A}
(v)\}]\leq\max\{\chi^{c}_{S}(x),\nu_{A}(y)\}\\
&  =\max\{0,\nu_{A}(y)\}=\nu_{A}(y).
\end{align*}

Consequently, $\mu_{A}(x\gamma y)\geq\mu_{A}(y)$ and $\nu_{A}(x\gamma y)\leq\nu_{A}(y)$. Hence $A=(\mu_{A},\nu_{A})$ is an $IFLI(S).$ Similarly we can prove in case of $IFRI(S).$\newline
\end{proof}

Using the above theorem we deduce the following theorem.

\begin{theorem}
In a $\Gamma$-semigroup $S$ the following are equivalent: $(1)$ $A=(\mu_{A},\nu_{A})$ is an $IFI(S),$ $(2)$ $S\circ A\subseteq A$ and $A\circ S\subseteq A,$ where $S=(\chi_{S},\chi^{c}_{S})$ and $\chi_{S}$ is the characteristic function of $S.$
\end{theorem}

\begin{proposition}
Let $A=(\mu_{A},\nu_{A})$ be an $IFRI(S)$ and $B=(\mu_{B},\nu_{B})$ be an $IFLI(S)$. Then $A\circ B\subseteq A\cap B.$
\end{proposition}

\begin{proof}
Let $A=(\mu_{A},\nu_{A})$ be an $IFRI(S)$ and $B=(\mu_{B},\nu_{B})$ be an $IFLI(S).$ Let $x\in S$. Suppose there exist $u_{1},v_{1}\in S$
and $\gamma_{1}\in\Gamma$ such that $x=u_{1}\gamma_{1}v_{1}$. Then%
\begin{align*}
(\mu_{A}\circ\mu_{B})(x)  &  =\underset{x=u\gamma v}{\sup}\min\{\mu_{A}%
(u),\mu_{B}(v)\}\\
&  \leq\underset{x=u\gamma v}{\sup}\min\{\mu_{A}(u\gamma v),\mu_{B}(u\gamma
v)\}\\
&  =\min\{\mu_{A}(x),\mu_{B})\}(x)=(\mu_{A}\cap\mu_{B})(x)
\end{align*}
and%
\begin{align*}
(\nu_{A}\circ\nu_{B})(x)  &  =\underset{x=u\gamma v}{\inf}\max\{\nu_{A}%
(u),\nu_{B}(v)\}\\
&  \geq\underset{x=u\gamma v}{\inf}\max\{\nu_{A}(u\gamma v),\nu_{B}(u\gamma
v)\}\\
&  =\max\{\nu_{A}(x),\nu_{B})\}(x)=(\nu_{A}\cup\nu_{B})(x).
\end{align*}
Suppose there do not exist $u,v\in S$ such that $x=u\gamma v$. Then $(\mu
_{A}\circ\mu_{B})(x)=0\leq(\mu_{A}\cap\mu_{B})(x)$ and $(\nu
_{A}\circ\nu_{B})(x)=1\geq(\nu_{A}\cup\nu_{B})(x).$ Hence the proof.\newline
\end{proof}

From the above proposition and the definition of $A\cap B$ the following proposition follows easily.

\begin{proposition}
Let $A=(\mu_{A},\mu_{A}),B=(\mu_{B},\nu_{B})\in IFI(S)$. Then $A\circ B\subseteq A\cap B\subseteq A,B.$
\end{proposition}

\begin{proposition}
Let $S$ be a regular $\Gamma$-semigroup and $A=(\mu_{A},\nu_{A})$ and $B=(\mu_{B},\nu_{B})$ be two $IFS(S)$. Then $A\circ B\supseteq A\cap B.$
\end{proposition}

\begin{proof}
Let $c\in S.$ Since $S$ is regular, then there exists an element $x\in S$ and
$\gamma_{1},\gamma_{2}\in\Gamma$ such that $c=c\gamma_{1}x\gamma_{2}c=c\gamma
c$ where $\gamma:=\gamma_{1}x\gamma_{2}\in\Gamma$. Then%
\begin{align*}
(\mu_{A}\circ\mu_{B})(c)  &  =\underset{c=u\alpha v}{\sup}\{\min\{\mu
_{A}(u),\mu_{B}(v)\}\}\\
&  \geq\min\{\mu_{A}(c),\mu_{B}(c)\}=(\mu_{A}\cap\mu_{B})(c)
\end{align*}
and
\begin{align*}
(\nu_{A}\circ\nu_{B})(c)  &  =\underset{c=u\alpha v}{\inf}\{\max\{\nu
_{A}(u),\nu_{B}(v)\}\}\\
&  \leq\max\{\nu_{A}(c),\nu_{B}(c)\}=(\nu_{A}\cup\nu_{B})(c).
\end{align*}
Hence $A\circ B\supseteq A\cap B.$
\end{proof}

To conclude this section we obtain the following characterization of a regular $\Gamma$-semigroup in terms of intutionistic fuzzy ideals.\\

\begin{theorem}
In a $\Gamma$-semigroup $S$ the following are equivalent: $(1)$ $S$ is
regular, $(2)$ $A\circ B=A\cap B$ where $A=(\mu_{A},\nu_{A})$ is an $IFRI(S)$ and $B=(\mu_{B},\nu_{B})$ is an $IFLI(S).$
\end{theorem}

\begin{proof}
Let $S$ be a regular $\Gamma$-semigroup. Then by Proposition $3.19$, $A
\circ B\supseteq A\cap B$. Again by Proposition $3.17$,
$A\circ B\subseteq A\cap B.$ Hence $A\circ B=A\cap B.$

Conversely, let $S$ be a $\Gamma$-semigroup and $A\circ B
=A\cap B$ where $A=(\mu_{A},\nu_{A})$ is an $IFRI(S)$ and $B=(\mu_{B},\nu_{B})$ is an $IFLI(S).$ Let $L$ and $R$ be respectively a $LI(S)$ and a $RI(S)$ and $x\in R\cap L.$ Then $x\in R$ and $x\in L.$ Hence
$(\chi_{L}(x),\chi^{c}_{L}(x))=(\chi_{R}(x),\chi^{c}_{R}(x))=(1,0)($where $\chi_{L}(x)$ and $\chi_{R}(x)$ are
respectively the characteristic functions of $L$ and $R$$)$. Thus%
\[
(\chi_{R}\cap\chi_{L})(x)=\min\{\chi_{R}(x),\chi_{L}(x)\}=1 \text{ and }
(\chi^{c}_{R}\cup\chi^{c}_{L})(x)=\max\{\chi^{c}_{R}(x),\chi^{c}_{L}(x)\}=0.
\]
Now by Corollary $3.13$, $(\chi_{L},\chi^{c}_{L})$ and $(\chi_{R},\chi^{c}_{R})$ are respectively an $IFLI(S)$ and an $IFRI(S).$ Hence by hypothesis, $\chi_{R}%
\circ\chi_{L}=\chi_{R}\cap\chi_{L}$ and $\chi^{c}_{R}\circ\chi^{c}_{L}=\chi^{c}_{R}\cup\chi^{c}_{L}$. Hence%
\begin{align*}
(\chi_{R}\circ\chi_{L})(x)  &  =1\\
i.e.,\underset{x=y\gamma z}{\sup}[\min\{\chi_{R}(y),\chi_{L}(z)\}  &  :y,z\in
S;\gamma\in\Gamma]=1
\end{align*} and

\begin{align*}
(\chi^{c}_{R}\circ\chi^{c}_{L})(x)  &  =0\\
i.e.,\underset{x=y\gamma z}{\inf}[\max\{\chi^{c}_{R}(y),\chi^{c}_{L}(z)\}  &  :y,z\in
S;\gamma\in\Gamma]=0
\end{align*}

This implies that there exist some $r,s\in S$ and $\gamma_{1}\in\Gamma$ such
that $x=r\gamma_{1}s$ and $(\chi_{R}(r),\chi^{c}_{R}(r))=(1,0)=(\chi_{L}(s),\chi^{c}_{L}(s))$. Hence $r\in R$ and $s\in L.$ Hence $x\in R\Gamma L.$ Thus $R\cap L\subseteq R\Gamma L.$ Also
$R\Gamma L\subseteq R\cap L.$ Hence $R\Gamma L=R\cap L.$ Consequently, the
$\Gamma$-semigroup $S$ is regular$(cf$. Theorem $1\cite{D1}).$
\end{proof}


\section{Intuitionistic Fuzzy Prime and Intuitionistic Fuzzy \\Semiprime Ideals}

\begin{definition}
An $IFI(S),$ $A=(\mu_{A},\nu_{A})$ is called an $IFPI(S)$\footnote{$PI(S),$ $SPI(S)$, $IFPI(S)$, $IFSPI(S)$ respectively denote the prime ideal(s), semiprime ideal(s), intuitionistic fuzzy prime ideal(s), intuitionistic fuzzy semiprime ideal(s) of a $\Gamma$-semigroup $S.$} if $\underset{\gamma\in\Gamma}{\inf}$ $\mu_{A}(x\gamma y)=\max\{\mu_{A}(x),\mu_{A}(y)\}\forall x,y\in S$ and $\underset{\gamma\in\Gamma}{\sup}$ $\nu_{A}(x\gamma y)=\min\{\nu_{A}(x),\nu_{A}(y)\}\forall x,y\in S.$
\end{definition}

\begin{definition}
An $IFI(S),$ $A=(\mu_{A},\nu_{A})$ is called an $IFSPI(S)$ if $\mu_{A}(x)\geq\underset{\gamma\in\Gamma}{\inf}$ $\mu_{A}(x\gamma x)\forall x\in S$ and $\nu_{A}(x)\leq\underset{\gamma\in\Gamma}{\sup}$ $\nu_{A}(x\gamma x)\forall x\in S.$
\end{definition}

By routine verification  we obtain the following proposition and two subsequent lemmas.

\begin{proposition}
If $\{A_{i}\}_{i\in \Lambda}$ is a family of $IFPI(S)(IFSPI(S)),$ then $\underset{i\in\Lambda}\bigcap A_{i}$ is an $IFPI(S)(IFSPI(S)).$
\end{proposition}

\begin{lemma}
If $A=(\mu_{A},\nu_{A})$ is an $IFPI(S)(IFSPI(S)),$ then so is $\square A=(\mu_{A},\mu^{c}_{A}).$
\end{lemma}

\begin{lemma}
If $A=(\mu_{A},\nu_{A})$ is an $IFPI(S)(IFSPI(S)),$ then so is $\lozenge A=(\nu^{c}_{A},\nu_{A}).$
\end{lemma}

Combining Lemmas $4.4$ and $4.5$ we obtain the following theorem.

\begin{theorem}
$A=(\mu_{A},\nu_{A})$ is an $IFPI(S)(IFSPI(S)),$ if and only if $\square A$ and $\lozenge A$ are $IFPI(S)(IFSPI(S)).$
\end{theorem}

\begin{theorem}
If $A=(\mu_{A},\nu_{A})$ is an $IFPI(S)(IFSPI(S)),$ then the upper and lower level cuts $U(\mu_{A};t)$ and $L(\nu_{A};t)$ are $PI(S)(SPI(S))$ for every $t\in Im(\mu_{A})\cap Im(\nu_{A}).$
\end{theorem}

\begin{proof}
Let $A=(\mu_{A},\nu_{A})$ be an $IFPI(S)$ and $t\in Im(\mu_{A})\cap Im(\nu_{A}).$ Let $x,y\in S$,$\gamma \in \Gamma$  and $x\gamma y\in U(\mu_{A};t).$ Then $\mu_{A}(x\gamma y)\geq t\forall\gamma\in\Gamma.$ So $\underset{\gamma\in\Gamma}{\inf}$ $\mu_{A}(x\gamma y)\geq t.$ It follows that $\max\{\mu_{A}(x),\mu_{A}(y)\}\geq t.$ So $\mu_{A}(x)\geq t$ or $\mu_{A}(y)\geq t.$ Hence $x\in U(\mu_{A};t)$ or $y\in U(\mu_{A};t).$ Hence $U(\mu_{A},t)$ is a $PI(S).$

Again let $x,y\in S$ and $x\gamma y\in L(\nu_{A};t).$ Then $\nu_{A}(x\gamma y)\leq t\forall\gamma\in\Gamma.$ So $\underset{\gamma\in\Gamma}{\sup}$ $\nu_{A}(x\gamma y)\leq t.$ It follows that $\min\{\nu_{A}(x),\nu_{A}(y)\}\leq t.$ So $\nu_{A}(x)\leq t$ or $\nu_{A}(y)\leq t.$ Hence $x\in L(\nu_{A};t)$ or $y\in L(\nu_{A};t).$ Hence $L(\nu_{A},t)$ is a $PI(S).$ Similarly we can prove the other case also.
\end{proof}

\begin{theorem}
If $A=(\mu_{A},\nu_{A})$ is an $IFS(S)$ such that the non-empty sets $U(\mu_{A};t)$ and $L(\nu_{A};t)$ are $PI(S)(SPI(S)),$ for $t\in [0,1].$ Then $A=(\mu_{A},\nu_{A})$ is an $IFPI(S)(IFSPI\\(S)).$
\end{theorem}

\begin{proof}
Let every non-empty upper level cut $U(\mu_{A};t)$ and $L(\nu_{A};t)$ are $PI(S).$ Let $x,y\in S.$ Let $\underset{\gamma\in\Gamma}{\inf}$ $\mu_{A}(x\gamma y)=t($we note here that $\mu_{A}(x\gamma y)\in [0,1]\forall\gamma\in\Gamma,\underset{\gamma\in\Gamma}{\inf}$ $\mu_{A}(x\gamma y)$ exists$).$ Then $\mu_{A}(x\gamma y)\geq t\forall\gamma\in\Gamma.$ So $x\gamma y\in U(\mu_{A};t)\forall\gamma\in\Gamma.$ So $U(\mu_{A};t)$ is non-empty and $x\Gamma y\subseteq U(\mu_{A};t).$ Since $U(\mu_{A};t)$ is a $PI(S),$ so $x\in U(\mu_{A};t)$ or $y\in U(\mu_{A};t).$ So $\mu_{A}(x)\geq t$ or $\mu_{A}(y)\geq t.$ So $\max\{\mu_{A}(x),\mu_{A}(y)\}\geq t,i.e.,\max\{\mu_{A}(x),\mu_{A}(y)\}\geq\underset{\gamma\in\Gamma}{\inf}$ $\mu_{A}(x\gamma y).....(1).$ Since $A=(\mu_{A},\nu_{A})$ is an $IFI(S),$ so $\mu_{A}(x\gamma y)\geq\max\{\mu_{A}(x),\mu_{A}(y)\}\forall\gamma\in\Gamma.$ Hence $\underset{\gamma\in\Gamma}{\inf}$ $\mu_{A}(x\gamma y)\geq\max\{\mu_{A}(x),\mu_{A}(y)\}......(2).$ Combining $(1)$ and $(2)$ we have $\underset{\gamma\in\Gamma}{\inf}$ $\mu_{A}(x\gamma y)=\max\{\mu_{A}(x),\mu_{A}(y)\}.$

Again let $x,y\in S.$ Let $\underset{\gamma\in\Gamma}{\sup}$ $\nu_{A}(x\gamma y)=t($we note here that $\nu_{A}(x\gamma y)\in [0,1]\forall\gamma\in\Gamma,\underset{\gamma\in\Gamma}{\sup}$ $\nu_{A}(x\gamma y)$ exists$).$ Then $\nu_{A}(x\gamma y)\leq t\forall\gamma\in\Gamma.$ So $x\gamma y\in L(\nu_{A};t)\forall\gamma\in\Gamma.$ So $L(\nu_{A};t)$ is non-empty and $x\Gamma y\subseteq L(\nu_{A};t).$ Since $L(\nu_{A};t)$ is a $PI(S),$ so $x\in L(\nu_{A};t)$ or $y\in L(\nu_{A};t).$ So $\nu_{A}(x)\leq t$ or $\nu_{A}(y)\leq t.$ So $\min\{\nu_{A}(x),\nu_{A}(y)\}\leq t,i.e.,\min\{\nu_{A}(x),\nu_{A}(y)\}\leq\underset{\gamma\in\Gamma}{\sup}$ $\nu_{A}(x\gamma y).....(3).$ Since $A=(\mu_{A},\nu_{A})$ is an $IFI(S),$ so $\nu_{A}(x\gamma y)\leq\min\{\nu_{A}(x),\nu_{A}(y)\}\forall\gamma\\\in\Gamma.$ Hence $\underset{\gamma\in\Gamma}{\sup}$ $\nu_{A}(x\gamma y)\leq\min\{\nu_{A}(x),\nu_{A}(y)\}......(4).$ Combining $(3)$ and $(4)$ we have $\underset{\gamma\in\Gamma}{\sup}$ $\nu_{A}(x\gamma y)=\min\{\nu_{A}(x),\nu_{A}(y)\}.$
Hence $A=(\mu_{A},\nu_{A})$ is an $IFPI(S).$ Similarly we can prove the other case also.
\end{proof}

By routine verification we obtain the following theorem.

\begin{theorem}
Let $I$ be a non-empty subset of $S.$ If two fuzzy
subsets $\mu$ and $\nu$ are defined on $S$ by $$\mu (x):=\left\{
\begin{array}{l}
\alpha _{0}\text{ \ if }x\in I \\
\alpha _{1}\text{ \ if }x\in S-I%
\end{array}%
\right. $$ and $$\nu (x):=\left\{
\begin{array}{l}
\beta _{0}\text{ \ if }x\in I \\
\beta _{1}\text{ \ if }x\in S-I%
\end{array}%
\right. $$
where $0\leq\alpha _{1}<\alpha _{0}, 0\leq\beta _{0}<\beta _{1}$ and $\alpha_{i}+\beta_{i}\leq 1$ for $i=0,1.$  Then $A=(\mu,\nu)$ is an $IFPI(S)(IFSPI(S)).$
\end{theorem}

Following result is the characteristic function criterion of an $IFPI(S)(IFSPI(S))$ which follows as an easy consequence of the above theorem.

\begin{corollary}
Let $\chi_{P}$ be the characteristic function of a $PI(S)(SPI(S)),$ $P.$ Then $I=(\chi_{P},\chi_{P}^{c})$ is $IFPI(S)(IFSPI(S)).$
\end{corollary}

\section{Intuitionistic Fuzzy Ideal Extension}

\begin{definition}
Let $S$ be a $\Gamma$-semigroup, $A=(\mu_{A},\nu_{A})$ be an $IFS(S)$ and $x\in S,$ then the $IFS(S),$ $<x,A>=\{(y,<x,\mu_{A}>(y),<x,\nu_{A}>(y)):x\in S\}$ where the functions $<x,\mu_{A}>:S\rightarrow [0,1]$ and $<x,\nu_{A}>:S\rightarrow [0,1]$ defined by $<x,\mu_{A}>(y):=\underset{\gamma\in\Gamma}{\inf}$ $\mu_{A}(x\gamma y)$ and $<x,\nu_{A}>(y):=\underset{\gamma\in\Gamma}{\sup}$ $\nu_{A}(x\gamma y)$ is called the $IFE(A)$\footnote{$IFE(A)$ denote the intuitionistic fuzzy extension of an $IFS$, $A=(\mu_{A},\nu_{A})$ of a $\Gamma$-semigroup $S.$} by $x.$
\end{definition}

\begin{example}
Let $S$ be the set of all non-positive integers and $\Gamma $ be the set of all non-positive even integers. Then $S$ is a $\Gamma $-semigroup where $a\gamma b$ and $\alpha a\beta$ denote the usual multiplication of integers $a,\gamma ,b$ and $\alpha,a,\beta$ respectively with $a,b\in S$ and $\alpha,\beta,\gamma \in \Gamma .$ Let $A=(\mu_{A},\nu_{A}) $ be an $IFS$ of $S,$ defined
as follows
$$
\mu_{A} (x)=\left\{
\begin{array}{ll}
1 & \text{if} \ x=0 \\
0.1 & \text{if} \ x=-1,-2 \\
0.2& \text{if} \ x<-2
\end{array}
\right.
$$
$$
\nu_{A} (x)=\left\{
\begin{array}{ll}
0 & \text{if} \ x=0 \\
0.6 & \text{if} \ x=-1,-2 \\
0.7 & \text{if} \ x<-2
\end{array}
\right. .
$$
For $x=0\in S,<x,\mu_{A}>(y)=1$ and $<x,\nu_{A}>(y)=0\forall y\in S.$ For all other $x\in S,<x,\mu_{A}>(y)=0.1$ and $<x,\nu_{A}>(y)=0.7\forall y\in S.$ Then the $IFS,$ $<x,A>=(<x,\mu_{A}>,<x,\nu_{A}>) $ of $S$ is an $IFE(A)$ by $x.$
\end{example}

\begin{proposition}
Let $A=(\mu_{A},\nu_{A})$ be an $IFI(S),$ where $S$ is commutative and $x\in S.$ Then $<x,A>=(<x,\mu_{A}>,<x,\nu_{A}>)$ is an $IFI(S).$
\end{proposition}

\begin{proof}
Let $A=(\mu_{A},\nu_{A})$ be an $IFI(S)$ where $S$ a commutative $\Gamma$-semigroup $S$ and $p,q\in S,\beta\in\Gamma.$ Then
$$
\begin{array}{ll}
<x,\mu_{A}>(p\beta q)& =\underset{\gamma\in\Gamma}{\inf}\mu_{A}(x\gamma p\beta q)\geq\underset{\gamma\in\Gamma}{\inf}\mu_{A}(x\gamma p)=<x,\mu_{A}>(p)
.
\end{array}
$$
Again
$$
\begin{array}{ll}
<x,\nu_{A}>(p\beta q)& =\underset{\gamma\in\Gamma}{\sup}\nu_{A}(x\gamma p\beta q)\leq\underset{\gamma\in\Gamma}{\sup}\nu_{A}(x\gamma p)=<x,\nu_{A}>(p)
.
\end{array}
$$
Thus $<x,A>$ is an $IFRI(S)$ Hence $S$ being commutative $<x,A>$ is an $IFI(S).$
\end{proof}

\begin{remark}
Commutativity of $\Gamma$-semigroup $S$ is not required to prove that $<x,A>$ is an $IFRI(S)$ when $A$ is an $IFRI(S).$
\end{remark}

\begin{proposition}
Let $A=(\mu_{A},\nu_{A})$ be an $IFPI(S)(IFSPI(S))$ where $S$ is commutative and $x\in S.$ Then $<x,A>=(<x,\mu_{A}>,<x,\nu_{A}>)$ is an $IFPI(S)(IFSPI(S)).$
\end{proposition}

\begin{proof}
Let $A=(\mu_{A},\nu_{A})$ be an $IFPI(S).$ Then by Proposition $5.2,$ $<x,A>$ is an $IFI(S).$ Let $y,z\in S.$ Then
$$
\begin{array}{ll}
\underset{\beta\in\Gamma}{\inf}<x,\mu_{A}>(y\beta z)& =\underset{\gamma\in\Gamma}{\inf}\underset{\gamma\in\Gamma}{\inf}\mu_{A}(x\gamma y\beta z)(cf. \text{ Definition 5.1)}\\
&=\underset{\beta\in\Gamma}{\inf}\max\{\mu_{A}(x),\mu_{A}(y\beta z)\}(cf. \text{ Definition 4.1)}\\
&=\max\{\mu_{A}(x),\underset{\beta\in\Gamma}{\inf}\mu_{A}(y\beta z)\}\\
&=\max\{\mu_{A}(x),\max\{\mu_{A}(y),\mu_{A}(z)\}\}\\
&=\max\{\max\{\mu_{A}(x),\mu_{A}(y)\},\max\{\mu_{A}(x),\mu_{A}(z)\}\}\\
&=\max\{\underset{\delta\in\Gamma}{\inf}\mu_{A}(x\delta y),\underset{\epsilon\in\Gamma}{\inf}\mu_{A}(x\epsilon z)\}\\
&=\max\{<x,\mu_{A}>(y),<x,\mu_{A}>(z)\}.
\end{array}
$$
Similarly we can show that $\underset{\beta\in\Gamma}{\sup}<x,\nu_{A}>(y\beta z)=\min\{<x,\nu_{A}>(y),<x,\nu_{A}>(z)\}.$ Hence by Definition $5.1,$ $<x,A>$ is an $IFPI(S).$ By routine calculation we can show that $<x,A>$ is an $IFSPI(S).$
\end{proof}

\begin{definition}
Let $S$ be a $\Gamma$-semigroup and $A=(\mu_{A},\nu_{A})$ be an $IFS(S).$ Then we define the support of $S$ by $Supp$ $ A=\{x\in S:\mu_{A}(x)>0$ and $\nu_{A}(x)<1\}.$
\end{definition}

\begin{proposition}
Let $S$ be a $\Gamma$-semigroup, $A=(\mu_{A},\nu_{A})$ be an $IFI(S)$ and $x\in S.$ Then we have the following:

$(1)$ $A\subseteq <x,A>.$

$(2)$ $<(x\alpha)^{n}x,A>\subseteq <(x\alpha)^{n+1}x,A>\forall\alpha\in\Gamma,\forall n\in N.$

$(3)$ If $\mu_{A}(x)>0$ and $\nu_{A}(x)<1$ then $Supp <x,A>=S.$
\end{proposition}

\begin{proof}
$(1)$ Let $y\in S.$ Then $<x,\mu_{A}>(y)=\underset{\gamma\in\Gamma}{\inf}\mu_{A}(x\gamma y)\geq\mu_{A}(y)($since $A$ is an $IFI(S)).$ Again $<x,\nu_{A}>(y)=\underset{\gamma\in\Gamma}{\sup}\nu_{A}(x\gamma y)\leq\nu_{A}(y)($since $A$ is an $IFI(S)).$ Hence $A\subseteq <x,A>.$

$(2)$ Let $y\in S.$ Then $<(x\alpha)^{n+1}x,\mu_{A}>(y)=\underset{\gamma\in\Gamma}{\inf}\mu_{A}((x\alpha)^{n+1}x\gamma y)=\underset{\gamma\in\Gamma}{\inf}\mu_{A}(x\alpha(x\alpha)^{n}x\\\gamma y)\geq\underset{\gamma\in\Gamma}{\inf}\mu_{A}((x\alpha)^{n}x\gamma y)($since $A$ is an $IFI(S))=<(x\alpha)^{n}x,\mu_{A}>(y).$ Similarly we can show that $<(x\alpha)^{n+1}x,\nu_{A}>(y)\leq<(x\alpha)^{n}x,\nu_{A}>(y).$ Hence $<(x\alpha)^{n}x,A>\subseteq <(x\alpha)^{n+1}x,A>.$

$(3)$ Since $<x,A>$ is an $IFS(S),$ so by definition $Supp <x,A>\subseteq S.$ Let $y\in S.$ Since $A$ is an $IFI(S),$ we have $<x,\mu_{A}>(y)=\underset{\gamma\in\Gamma}{\inf}\mu_{A}(x\gamma y)\geq\mu_{A}(x)>0$ and $<x,\nu_{A}>(y)=\underset{\gamma\in\Gamma}{\sup}\nu_{A}(x\gamma y)\leq\nu_{A}(x)<1.$ So $y\in Supp<x,A>.$ Consequently, $S\subseteq Supp<x,A>.$ Hence $Supp <x,A>=S.$
\end{proof}

\begin{remark}
If $(x\alpha)^{0}x=x$ then Proposition $5.5(2)$ is also true for $n=0.$
\end{remark}

\begin{definition}
$\cite{D3}$ Let $S$ be a $\Gamma$-semigroup, $A\subseteq S$ and $x\in S.$ We define $<x,A>=\{y\in S:x\Gamma y\subseteq A\},$ where $x\Gamma y:=\{x\alpha y:\alpha\in\Gamma\}.$
\end{definition}

\begin{proposition}
Let $S$ be a $\Gamma$-semigroup, $\phi\neq A\subseteq S$ and $\chi_{A}$ be the characteristic function of $A.$ Then for all $x\in S,$ $<x,\chi_{A}>=\chi_{<x,A>}$ and $<x,\chi^{c}_{A}>=\chi^{c}_{<x,A>}.$
\end{proposition}

\begin{proof}
Let $x,y\in S.$ If $y\in <x,A>$ then $x\Gamma y\subseteq A,$ which implies that $x\gamma y\in A\forall\gamma\in\Gamma.$ Then $\chi_{A}(x\gamma y)=1\forall\gamma\in\Gamma$ and hence $\underset{\gamma\in\Gamma}{\inf}\chi_{A}(x\gamma y)=1$ whence $<x,\chi_{A}>(y)=1.$ Also $\chi_{<x,A>}(y)=1.$ Again if $y\notin <x,A>$ then there exists $\gamma\in\Gamma$ such that $x\gamma y\notin A.$ So $\chi_{A}(x\gamma y)=0.$ Hence $\underset{\gamma\in\Gamma}{\inf}\chi_{A}(x\gamma y)=0.$ Thus $\underset{\gamma\in\Gamma}{\inf}\chi_{A}(x\gamma y)=0$ whence $<x,\chi_{A}>(y)=0.$ Also $\chi_{<x,A>}(y)=0.$ Hence $<x,\chi_{A}>=\chi_{<x,A>}.$ Similarly by routine calculation we can show that $<x,\chi^{c}_{A}>=\chi^{c}_{<x,A>}.$
\end{proof}

\begin{proposition}
Let $S$ be a $\Gamma$-semigroup and $A=(\mu_{A},\nu_{A})$ be a non-empty $IFS(S).$ Then for any $t\in [0,1],$ $<x,U(\mu_{A};t)>=U(<x,\mu_{A}>;t)$ and $<x,L(\nu_{A};t)>=L(<x,\nu_{A}>;t)$ for all $x\in S.$
\end{proposition}

\begin{proof}
Let $y\in U(<x,\mu_{A}>;t).$ Then $<x,\mu_{A}>(y)\geq t.$ Hence $\underset{\gamma\in\Gamma}{\inf}\mu_{A}(x\gamma y)\geq t.$ This gives $\mu_{A}(x\gamma y)\geq t\forall\gamma\in\Gamma$ and hence $x\gamma y\in U(\mu_{A};t).$ Consequently, $y\in <x,U(\mu_{A};t)>.$ It follows that $U(<x,\mu_{A}>;t)\subseteq <x,U(\mu_{A};t)>.$ Reversing the above argument we can deduce that $<x,U(\mu_{A};t)>\subseteq U(<x,\mu_{A}>;t).$ Hence $<x,U(\mu_{A};t)>=U(<x,\mu_{A}>;t)\forall x\in S.$ Similarly by routine calculation we can show that $<x,L(\nu_{A};t)>=L(<x,\nu_{A}>;t)\forall x\in S.$
\end{proof}

\begin{proposition}
Let $S$ be a commutative $\Gamma$-semigroup and $A=(\mu_{A},\nu_{A})$ be an $IFS(S)$ such that $<x,A>=A$ for every $x\in S.$ Then $A=(\mu_{A},\nu_{A})$ is a constant function.
\end{proposition}

\begin{proof}
Let $x,y\in S.$ Let $<x,A>=A\forall x\in S.$ Then $<x,\mu_{A}>=\mu_{A}$ and $<x,\nu_{A}>=\nu_{A}\forall x\in S.$ Then by hypothesis we have $\mu_{A}(x)=<y,\mu_{A}>(x)=\underset{\gamma\in\Gamma}{\inf}\mu_{A}(y\gamma x)=\underset{\gamma\in\Gamma}{\inf}\mu_{A}(x\gamma y)($since $S$ is commutative$)=<x,\mu_{A}>(y)=\mu_{A}(y).$ Again $\nu_{A}(x)=<y,\nu_{A}>(x)=\underset{\gamma\in\Gamma}{\sup}\nu_{A}(y\gamma x)=\underset{\gamma\in\Gamma}{\sup}\nu_{A}(x\gamma y)($since $S$ is commutative$)=<x,\nu_{A}>(y)=\nu_{A}(y).$ Hence $A=(\mu_{A},\nu_{A})$ is a constant function.
\end{proof}

\begin{corollary}
Let $S$ be a commutative $\Gamma$-semigroup and $A=(\mu_{A},\nu_{A})$ be an $IFPI(S).$ If $A=(\mu_{A},\nu_{A})$ is not constant, then $A=(\mu_{A},\nu_{A})$ is not maximal $IFPI(S).$
\end{corollary}

\begin{proof}
Let $A=(\mu_{A},\nu_{A})$ be an $IFPI(S).$ Since $S$ is commutative, by Proposition $5.3,$ $<x,A>$ is an $IFPI(S).$ Now by Proposition $5.5(1),$ $A\subseteq <x,A>\forall x\in S.$ If $A=<x,A>$ for all $x\in S$ then by Proposition $5.9,$ $A$ is a constant function which contradicts the hypothesis. Hence there exists $x\in S$ such that $A\subsetneq <x,A>.$ This completes the proof.
\end{proof}

\begin{corollary}
Let $S$ be a commutative $\Gamma$-semigroup, $\{A_{i}\}_{i\in I}=(\mu_{A_{i}},\nu_{A_{i}})_{i\in I}$ be a non-empty family of $IFSPI(S)$ and $A=(\mu_{A},\nu_{A})=\underset{i\in I}{\inf}$ $ A_{i}=(\underset{i\in I}{\inf}$ $\mu_{A_{i}},\underset{i\in I}{\sup}$ $\nu_{A_{i}}).$ Then for any $x\in S,$ $<x,A>=(<x,\mu_{A}>,<x,\nu_{A}>)$ is an $IFSPI(S).$
\end{corollary}

\begin{proof}
Since each $A_{i}=(\mu_{A_{i}},\nu_{A_{i}})(i\in I)$ is an $IFI(S),$ $\mu_{A_{i}}(0)\neq 0$ and $\nu_{A_{i}}\neq 1\forall i\in I($each $\mu_{A_{i}}$ and $\nu_{A_{i}}$ are non-empty, so there exist $x_{i}\in S$ such that $\mu_{A_{i}}(x_{i})\neq 0$ and $\nu_{A_{i}}(x_{i})\neq 1\forall i\in I.$ Also $\mu_{A_{i}}(0)=\mu_{A_{i}}(0\gamma x_{A_{i}})\geq\mu_{A_{i}}(x_{i})$ and $\nu_{A_{i}}(0)=\nu_{A_{i}}(0\gamma x_{A_{i}})\leq\nu_{A_{i}}(x_{i})\forall i\in I.$ Hence $\forall i\in I,\mu_{A_{i}}(0)\neq 0$ and $\nu_{A_{i}}(0)\neq 1).$ Consequently, $\mu_{A}(0)\neq 0$ and $\nu_{A}(0)\neq 1.$ Thus $A=(\mu_{A},\nu_{A})$ is non-empty. Now let $x,y\in S$ and $\gamma\in\Gamma.$ Then $\mu_{A}(x\gamma y)=\underset{i\in I}{\inf}$ $\mu_{A_{i}}(x\gamma y)\geq\underset{i\in I}{\inf}$ $\mu_{A_{i}}(x)=\mu_{A}(x)$ and $\nu_{A}(x\gamma y)=\underset{i\in I}{\sup}$ $\nu_{A_{i}}(x\gamma y)\leq\underset{i\in I}{\sup}$ $\nu_{A_{i}}(x)=\nu_{A}(x).$ Hence $S$ being commutative $A$ is an $IFI(S).$

Again let $a\in S.$ Then $\mu_{A}(a)=\underset{i\in I}{\inf}$ $\mu_{A_{i}}(a)\geq\underset{i\in I}{\inf}$ $\underset{\gamma\in\Gamma}{\inf}$ $\mu_{A_{i}}(a\gamma a)=\underset{\gamma\in\Gamma}{\inf}$ $\underset{i\in I}{\inf}$ $\mu_{A_{i}}(a\gamma a)=\underset{\gamma\in\Gamma}{\inf}$ $\mu_{A}(a\gamma a)$ and $\nu_{A}(a)=\underset{i\in I}{\sup}$ $\nu_{A_{i}}(a)\leq\underset{i\in I}{\sup}$ $\underset{\gamma\in\Gamma}{\sup}$ $\nu_{A_{i}}(a\gamma a)=\underset{\gamma\in\Gamma}{\sup}$ $\underset{i\in I}{\sup}$ $\nu_{A_{i}}(a\gamma a)=\underset{\gamma\in\Gamma}{\sup}$ $\nu_{A}(a\gamma a).$ Thus $A=(\mu_{A},\nu_{A})$ is an $IFSPI(S).$ Hence by Proposition $5.3,$ $<x,A>$ is an $IFSPI(S).$
\end{proof}

\begin{remark}
The above corollary shows that in a $\Gamma$-semigroup intersection of arbitrary family of $IFSPI(S)$ is an $IFSPI(S).$
\end{remark}

\begin{corollary}
Let $S$ be a commutative $\Gamma$-semigroup, $\{S_{i}\}_{i}$ be a non-empty family of $SPI(S),$  $A:=\underset{i\in I}{\bigcap} S_{i}\neq\phi$ and $M=(\mu_{A},\mu^{c}_{A})$ where $\mu_{A}$ is the characteristic function of $A.$ Then $<x,M>$ is an $IFSPI(S),$ for all $x\in S.$
\end{corollary}

\begin{proof}
By the given condition we have $A=\underset{i\in I}{\bigcap} S_{i}\neq\phi.$ Hence $A=(\mu_{A},\nu_{A})$ is non-empty. Let $x\in S.$ If $x\in A,$ then $\mu_{A}(x)=1,\mu^{c}_{A}(x)=0$ and $x\in S_{i}\forall i\in I.$ Hence $\underset{i\in I}{\inf}$ $\mu_{S_{i}}(x)=1=\mu_{A}(x)$ and $\underset{i\in I}{\sup}$ $\mu^{c}_{S_{i}}(x)=0=\mu^{c}_{A}(x).$ If $x\notin A,$ then $\mu_{A}(x)=0,\mu^{c}_{A}(x)=1$ and for some $i\in I,x\notin S_{i}.$ It follows that $\underset{i\in I}{\inf}$ $\mu_{S_{i}}(x)=0=\mu_{A}(x)$ and $\underset{i\in I}{\sup}$ $\mu^{c}_{S_{i}}(x)=1=\mu^{c}_{A}(x).$ Thus we see that $M=(\mu_{A},\mu^{c}_{A})=(\underset{i\in I}{\inf}$ $\mu_{S_{i}},\underset{i\in I}{\sup}$ $\mu^{c}_{S_{i}}).$ Again $(\mu_{S_{i}},\mu^{c}_{S_{i}})$ is an $IFSPI(S)$ for all $i\in I.$ Hence by Corollary $5.11,$ for all $x\in S,$ $<x,M>$ is an $IFSPI(S).$
\end{proof}

\begin{theorem}
Let $S$ be a $\Gamma$-semigroup. If $A=(\mu_{A},\nu_{A})$ is an $IFPI(S)$ and $x\in S$ such that $\mu_{A}(x)=\underset{y\in S}{\inf}$ $\mu_{A}(y)$ and $\nu_{A}(x)=\underset{y\in S}{\sup}$ $\nu_{A}(y),$ then $<x,A>=A.$ Conversely, if $A=(\mu_{A},\nu_{A})$ be an $IFI(S)$ such that $<y,A>=A\forall y\in S$ with $\mu_{A}(y)$ is not maximal in $\mu_{A}(S)$ and $\nu_{A}(y)$ is not minimal in $\nu_{A}(S)$ then $A=(\mu_{A},\nu_{A})$ is $IFPI(S).$
\end{theorem}

\begin{proof}
Let $A=(\mu_{A},\nu_{A})$ be an $IFPI(S)$ and $x\in S$ such that $\mu_{A}(x)=\underset{y\in S}{\inf}$ $\mu_{A}(y)$ and $\nu_{A}(x)=\underset{y\in S}{\sup}$ $\nu_{A}(y)($since each $\mu_{A}(y)$ and each $\nu_{A}(y)\in [0,1],$ a closed and bounded subset of $R,$ so $\underset{y\in S}{\inf}$ $\mu_{A}(y)$ and $\underset{y\in S}{\sup}$ $\nu_{A}(y)$ exist$).$ Let $z\in S.$ Then $\mu_{A}(x)\leq\mu_{A}(z)$ and $\nu_{A}(x)\geq\nu_{A}(z).$ Hence $\max\{\mu_{A}(x),\mu_{A}(z)\}=\mu_{A}(z).......(*)$ and $\min\{\nu_{A}(x),\nu_{A}(z)\}=\nu_{A}(z)......(**).$ Then
$$
\begin{array}{ll}
<x,\mu_{A}>(z)&=\underset{\gamma\in\Gamma}{\inf}\mu_{A}(x\gamma z)\\
&=\max\{\mu_{A}(x),\mu_{A}(z)\}(\text{since }A=(\mu_{A},\nu_{A})\text{ is an }IFPI(S))\\
&=\mu_{A}(z)(\text{ using }(*))
\end{array}
$$
and
$$
\begin{array}{ll}
<x,\nu_{A}>(z)&=\underset{\gamma\in\Gamma}{\sup}\nu_{A}(x\gamma z)\\
&=\min\{\nu_{A}(x),\nu_{A}(z)\}(\text{since }A=(\mu_{A},\nu_{A})\text{ is an }IFPI(S))\\
&=\nu_{A}(z)(\text{ using }(**)).
\end{array}
$$
Consequently, $<X,A>=A.$

Conversely, let $A=(\mu_{A},\nu_{A})$ be an $IFI(S)$ such that $<y,A>=A\forall y\in S$ with $\mu_{A}(y)$ is not maximal in $\mu_{A}(S)$ and $\nu_{A}(y)$ is not minimal in $\nu_{A}(S)$ and let $x_{1},x_{2}\in S.$ Since $A=(\mu_{A},\nu_{A})$ is an $IFI(S),$ so we have $\mu_{A}(x_{1}\gamma x_{2})\geq\max\{\mu_{A}(x_{1}),\mu_{A}(x_{2})\}\forall\gamma\in\Gamma$ and $\nu_{A}(x_{1}\gamma x_{2})\leq\min\{\nu_{A}(x_{1}),\nu_{A}(x_{2})\}\forall\gamma\in\Gamma.$ This leads to
$$
\begin{array}{ll}
\underset{\gamma\in\Gamma}{\inf}\mu_{A}(x_{1}\gamma x_{2})\geq\max\{\mu_{A}(x_{1}),\mu_{A}(x_{2})\}
\end{array}
$$ and
$$
\begin{array}{ll}
\underset{\gamma\in\Gamma}{\sup}\nu_{A}(x_{1}\gamma x_{2})\leq\min\{\nu_{A}(x_{1}),\nu_{A}(x_{2})\}\\
\ \ \ \ \ \ \ \ \ \ \ \ \ \ \ \ \ \ \ \ \ \ \ \ \ \ \ \ \ \ \ \ \ \ \ \ \ \ ...(***).
\end{array}
$$
Now let us consider the following two cases:
$$
\begin{array}{ll}
(i)\text{ Either }\mu_{A}(x_{1})\text{ or }\mu_{A}(x_{2})\text{ is maximal in }\mu_{A}(S)\text{ and either }\nu_{A}(x_{1})\text{ or }\nu_{A}(x_{2})\\
\text{ is minimal in }\nu_{A}(S).
\end{array}
$$
$$
\begin{array}{ll}
(ii)\text{ Neither }\mu_{A}(x_{1})\text{ nor }\mu_{A}(x_{2})\text{ is maximal in }\mu_{A}(S)\text{ and neither }\nu_{A}(x_{1})\text{ nor }\\
\nu_{A}(x_{2})\text{ is minimal in }\nu_{A}(S).
\end{array}
$$
Case $(i):$ Without loss of generality, let $\mu_{A}(x_{1})$ is maximal in $\mu_{A}(S)$ and $\nu_{A}(x_{1})$ is minimal in $\nu_{A}(S).$ Then
$$
\begin{array}{ll}
\underset{\gamma\in\Gamma}{\inf}$ $\mu_{A}(x_{1}\gamma x_{2})\leq\mu_{A}(x_{1})=\max\{\mu_{A}(x_{1}),\mu_{A}(x_{2})\}
\end{array}
$$ and
$$
\begin{array}{ll}
\underset{\gamma\in\Gamma}{\sup}$ $\nu_{A}(x_{1}\gamma x_{2})\geq\nu_{A}(x_{1})=\min\{\nu_{A}(x_{1}),\nu_{A}(x_{2})\}\\
\ \ \ \ \ \ \ \ \ \ \ \ \ \ \ \ \ \ \ \ \ \ \ \ \ \ \ \ \ \ \ \ \ \ \ \ \ \ \ \ \ \ \ \ \ .....(****).
\end{array}
$$
Thus by $(***)$ and $(****)$ we have $\underset{\gamma\in\Gamma}{\inf}$ $\mu_{A}(x_{1}\gamma x_{2})=\max\{\mu_{A}(x_{1}),\mu_{A}(x_{2})\}$ and $\underset{\gamma\in\Gamma}{\sup}$ $\nu_{A}(x_{1}\gamma x_{2})=\min\{\nu_{A}(x_{1}),\nu_{A}(x_{2})\}.$

Case $(ii):$ By hypothesis $<x_{1},\mu_{A}>=<x_{2},\mu_{A}>=\mu_{A}$ and $<x_{1},\nu_{A}>=<x_{2},\nu_{A}>=\nu_{A}.$ Hence $<x_{1},\mu_{A}>(x_{2})=\mu_{A}(x_{2})\Rightarrow\underset{\gamma\in\Gamma}{\inf}$ $\mu_{A}(x_{1}\gamma x_{2})=\mu_{A}(x_{2})=\max\{\mu_{A}(x_{1}),\mu_{A}(x_{2}\}$ and $<x_{1},\nu_{A}>(x_{2})=\nu_{A}(x_{2})\Rightarrow\underset{\gamma\in\Gamma}{\sup}$ $\nu_{A}(x_{1}\gamma x_{2})=\nu_{A}(x_{2})=\min\{\nu_{A}(x_{1}),\nu_{A}(x_{2}\}.$ Hence $A=(\mu_{A},\nu_{A})$ is an $IFPI(S).$

\end{proof}

To end this section and to conclude the paper we deduce the following characterization of a $PI(S)$ which follows as a corollary to the above theorem.

\begin{corollary}
Let $S$ be a $\Gamma$-semigroup, $I$ be an $I(S)$ and $M=(\mu_{I},\mu^{c}_{I})$ where $\mu_{I}$ is the characteristic function of $I.$ Then $I$ is $PI(S)$ if and only if for $x\in S$ with $x\notin I,$ $<x,M>=M.$
\end{corollary}

\begin{proof}
Let $I$ be a $PI(S).$ Then by Corollary $4.10,$ $M=(\mu_{I},\mu^{c}_{I})$ is an $IFPI(S).$ Let $x\in S$ such that $x\notin I,$ then $\mu_{I}(x)=0=\underset{y\in S}{\inf}$ $\mu_{I}(y)$ and $\mu^{c}_{I}(x)=1=\underset{y\in S}{\sup}$ $\mu^{c}_{I}(y).$ Hence by Theorem $5.13,$ $<x,M>=M.$

Conversely, let $<x,M>=M$ for all $x\in S$ with $x\notin I.$ Let $y\in S$ be such that $\mu_{A}(y)$ is not maximal in $\mu_{A}(S)$ and $\nu_{A}(y)$ is not minimal in $\nu_{A}(S).$ Then $\mu_{I}(y)=0$ and $\mu^{c}_{I}(y)=1$ and so $y\notin I.$ So $<y,M>=M.$ So by Theorem $5.13,$ $M=(\mu_{I},\mu^{c}_{I})$ is an $IFPI(S).$ Hence $I$ is a $PI(S)(cf.$ Corollary $4.10).$
\end{proof}


\end{document}